\documentclass[11pt,reqno]{amsart}

\usepackage{amssymb,latexsym,amsmath,epsfig,mathrsfs,mathabx}
\usepackage[all]{xy}
\usepackage{graphicx}
\usepackage{float}
\usepackage{textcomp}
\usepackage{tikz}
\usepackage{hyperref}
\usepackage{url}
\usepackage{color}
\usetikzlibrary{matrix,arrows}

\usepackage[letterpaper,margin=1in]{geometry}

\def\NN{\mathbb N}
\def\ZZ{\mathbb Z}
\def\QQ{\mathbb Q}
\def\RR{\mathbb R}

\def\Ql{\mathbb{Q}_\ell}

\def\F{\mathbb F}

\def\P{\mathbb P}

\def\OK{\mathcal{O}_K}

\def\fp{\mathfrak p}

\newcommand{\frob}{\operatorname{Frob}}

\theoremstyle{plain}
\numberwithin{equation}{section}
\newtheorem{theorem}{Theorem}[section]
\newtheorem{cor}[theorem]{Corollary}
\newtheorem{prop}[theorem]{Proposition}
\newtheorem{conjecture}[theorem]{Conjeture}
\newtheorem{lemma}[theorem]{Lemma}
\newtheorem{mydef}[theorem]{Definition}
\newtheorem{remark}[theorem]{Remark}
\newtheorem{example}[theorem]{Example}

\numberwithin{equation}{section}
\numberwithin{theorem}{section}

\def\Hone{H^1 _{\acute{e}t} (X, \overline{\mathbb Q_{\ell}})}
\def\Hzero{H^0 _{\acute{e}t} (X, \overline{\mathbb Q_{\ell}})}
\def\Htwo{H^2 _{\acute{e}t} (X, \overline{\mathbb Q_{\ell}})}

\def\HiX{H^i _{\acute{e}t} (X, \overline{\mathbb Q_{\ell}})}
\def\HizeroX{H^{i_0} _{\acute{e}t} (X, \overline{\mathbb Q_{\ell}})}
\def\HstarX{H^{*} _{\acute{e}t} (X, \overline{\mathbb Q_{\ell}})}

\begin{document}

\title{Newman's Conjecture in Function Fields}

\author[Chang]{Alan Chang}
\email{\textcolor{blue}{\href{mailto:ac@math.uchicago.edu}{ac@math.uchicago.edu}}}
\address{Department of Mathematics, University of Chicago, Chicago, IL 60637}

\author[Mehrle]{David Mehrle}
\email{\textcolor{blue}{\href{mailto:dmehrle@cmu.edu}{dmehrle@cmu.edu}}}
\address{Department of Mathematical Sciences, Carnegie Mellon University, Pittsburgh, PA 15289}

\author[Miller]{Steven J. Miller}
\email{\textcolor{blue}{\href{mailto:sjm1@williams.edu}{sjm1@williams.edu}},  \textcolor{blue}{\href{Steven.Miller.MC.96@aya.yale.edu}{Steven.Miller.MC.96@aya.yale.edu}}}
\address{Department of Mathematics and Statistics, Williams College, Williamstown, MA 01267}

\author[Reiter]{Tomer Reiter}
\email{\textcolor{blue}{\href{mailto:treiter@andrew.cmu.edu}{treiter@andrew.cmu.edu}}}
\address{Department of Mathematical Sciences, Carnegie Mellon University, Pittsburgh, PA 15289}

\author[Stahl]{Joseph Stahl}
\email{\textcolor{blue}{\href{mailto:jstahl@bu.edu}{jstahl@bu.edu}}}
\address{Department of Mathematics \& Statistics, Boston University, Boston, MA 02215}

\author[Yott]{Dylan Yott}
\email{\textcolor{blue}{\href{mailto:dyott@math.berkeley.edu}{dyott@math.berkeley.edu}}}
\address{Department of Mathematics, University of California, Berkeley, Berkeley, CA 94720}

%\green{REMOVED 11M50: Connections to random matrices. OTHERWISE SEEMS GOOD.}


\date{\today}

\thanks{This work was supported by NSF grants DMS-1347804, DMS-1265673, Williams College, and the PROMYS program. The authors thank Noam Elkies, David Geraghty, Rob Pollack, Glenn Stevens, and Keith Conrad for their insightful comments and support.}

%%%%%%%%%%%%%%%%%%%%%%%%%%%%%%%%%%%%%%%%%%%%%%%%%%%%%%%%%%%%%%%%%%%%%%%%%%%%%%%%%%%%%%%%%%%%%%%%%%%%%%%%%%%%%%%%%%%%%%%%%%%%%%%%%%%%
% ABSTRACT %
%%%%%%%%%%%%%%%%%%%%%%%%%%%%%%%%%%%%%%%%%%%%%%%%%%%%%%%%%%%%%%%%%%%%%%%%%%%%%%%%%%%%%%%%%%%%%%%%%%%%%%%%%%%%%%%%%%%%%%%%%%%%%%%%%%%%

\begin{abstract}
De Bruijn and Newman introduced a deformation of the completed Riemann zeta function $\zeta$, and proved there is a real constant $\Lambda$ which encodes the movement of the nontrivial zeros of $\zeta$ under the deformation. The Riemann hypothesis is equivalent to the assertion that $\Lambda\leq 0$. Newman, however, conjectured that $\Lambda\geq 0$, remarking, ``the new conjecture is a quantitative version of the dictum that the Riemann hypothesis, if true, is only barely so.'' Andrade, Chang and Miller extended the machinery developed by Newman and Polya to $L$-functions for function fields. In this setting we must consider a modified Newman's conjecture: $\sup_{f\in\mathcal{F}} \Lambda_f \geq 0$, for $\mathcal{F}$ a family of $L$-functions.

We extend their results by proving this modified Newman's conjecture for several families of $L$-functions. In contrast with previous work, we are able to exhibit specific $L$-functions for which $\Lambda_D = 0$, and thereby prove a stronger statement: $\max_{L\in\mathcal{F}} \Lambda_L = 0$. Using geometric techniques, we show a certain deformed $L$-function must have a double root, which implies
$\Lambda = 0$. For a different family, we construct particular elliptic curves with $p + 1$ points over $\mathbb{F}_p$.
By the Weil conjectures, this has either the maximum or minimum possible number of points over $\mathbb{F}_{p^{2n}}$. The fact that $\#E(\mathbb{F}_{p^{2n}})$ attains the bound tells us that the associated $L$-function satisfies $\Lambda = 0$.
\end{abstract}

\maketitle

%\green{REORGANIZED THE INTRODUCTION, COMBINED WITH THE SECTION ON CLASSICAL NEWMAN'S CONJECTURE.}

%%%%%%%%%%%%%%%%%%%%%%%%%%%%%%%%%%%%%%%%%%%%%%%%%%%%%%%%%%%%%%%%%%%%%%%%%%%%%%%%%%%%%%%%%%%%%%%%%%%%%%%%%%%%%%%%%%%%%%%%%%%%%%%%%%%%
% INTRODUCTION %
%%%%%%%%%%%%%%%%%%%%%%%%%%%%%%%%%%%%%%%%%%%%%%%%%%%%%%%%%%%%%%%%%%%%%%%%%%%%%%%%%%%%%%%%%%%%%%%%%%%%%%%%%%%%%%%%%%%%%%%%%%%%%%%%%%%%
\section{Introduction}

Newman's conjecture, as originally formulated, is a statement about the zeros of a deformation of the completed Riemann zeta function. This deformation was introduced by P\'{o}lya to attack the Riemann hypothesis, but Newman's conjecture regarding this deformation is in fact an almost counter-conjecture to the Riemann hypothesis. The classical Newman's conjecture is explained below in Section \ref{sec:classicalnewman}.

%%%%%%%%%%%%%%%%%%%%%%%%%%%%%%%%%%%%%%%%%%%%%%%%%%%%%%%%%%%%%%%%%%%%%%%%%%%%%%%%%%%%%%%%%%%%%%%%%%%%%%%%
%SUBSECTION: THE CLASSICAL NEWMAN'S CONJECTURE
%%%%%%%%%%%%%%%%%%%%%%%%%%%%%%%%%%%%%%%%%%%%%%%%%%%%%%%%%%%%%%%%%%%%%%%%%%%%%%%%%%%%%%%%%%%%%%%%%%%%%%%%
\subsection{The Classical Newman's Conjecture}\label{sec:classicalnewman}

Instead of working with the Riemann zeta function $\zeta(s)$ itself, define the completed Riemann zeta function
\begin{equation}
\label{eq:completedzeta}
\xi(s)\ :=\ \frac{s(s-1)}{2}\, \pi^{-s/2}\, \Gamma \left( \frac{s}{2} \right) \zeta(s).
\end{equation}
This has the effect of eliminating the trivial zeros of $\zeta(s)$, but keeping all of the nontrivial ones. Additionally, the functional equation for $\xi(s)$  is simpler than that for the Riemann zeta function: $\xi(s) = \xi(1-s)$. To simplify the analysis further, introduce
\begin{equation}
\label{eq:changevariableszeta}
 \Xi(x)\ :=\ \xi\left( \frac{1}{2} + i x \right)
\end{equation}
to shift the zeros of $\xi(s)$ to lie along the real line. Recall that because $\xi$ is analytic and real valued on the real line, then $\xi(\overline{s}) = \overline{\xi(s)}$. Combining this fact and the functional equation for $\xi$, we have that $\Xi(x) \in \RR$ whenever $x \in \RR$. Now let $\Phi(u)$ denote the Fourier transform\footnote{We normalize so that the Fourier transform of $f(x)$ is $\int_{-\infty}^\infty f(x) e^{-ixu}du$.} of $\Xi(x)$. Because $\Xi(x)$ decays rapidly as $x \to \infty$, we may introduce a ``time'' parameter into the inverse Fourier transform.

\begin{mydef}
The \textit{deformed Riemann zeta function} $\Xi_t$ is
\begin{equation}
\Xi_t(x)\ :=\ \int_0^\infty\!e^{tu^2}\Phi(u)\left(e^{iux} + e^{-iux}\right)\,du.
\end{equation}
\end{mydef}

Note that $\Xi_0(x) = \Xi(x)= \xi(1/2 + i x)$ agrees with \eqref{eq:changevariableszeta}. This deformation $\Xi_t(s)$ is the function that P\'{o}lya hoped to use to attack the Riemann Hypothesis, because the Riemann Hypothesis is equivalent to the statement that all of the zeros of $\Xi_0(x)$ are real. De Bruijn managed to prove a related statement.

\begin{lemma}[{De Bruijn \cite[Theorem 13]{dB}}]
\label{lem:realzeros}
If $t \in \RR$ is such that $\Xi_t$ has only real zeros, then for all $t' \geq t$, $\Xi_{t'}$ has only real zeros.
\end{lemma}

P\'{o}lya wanted to show that $\Xi_t$ has only real zeros for all $t \in \RR$, which would imply the Riemann Hypothesis. Unfortunately, this is not the case, as the next lemma shows.

\begin{lemma}[{Newman, \cite[Theorem 3]{New}}]
\label{lem:nonrealzero} There is some $t \in \RR$ such that $\Xi_t$ has a nonreal zero.
\end{lemma}

However, it is possible to place bounds on this $t$; this is how Newman salvaged P\'{o}lya's strategy. By combining the previous two lemmas, we may define the De Bruijn--Newman Constant.

\begin{mydef} The \emph{De Bruijn--Newman Constant} $\Lambda \in \RR$ is the value such that
\vspace*{-0.2em}
\begin{itemize}
\itemsep0em
\item if $t \geq \Lambda$, then $\Xi_t$ has only real zeros;
\item if $t < \Lambda$, then $\Xi_t$ has a non-real zero.
\end{itemize}
\vspace*{-0.2em}
Such a constant exists because of \eqref{lem:realzeros} and \eqref{lem:nonrealzero}.
\end{mydef}

This allows us to rephrase the Riemann hypothesis yet again. The Riemann Hypothesis is true if and only if $\Lambda \leq 0$, that is, if and only if $\Xi_0$ has only real zeros. On the other hand, Newman made the following conjecture.

\begin{conjecture}[{Newman \cite[Remark 2]{New}}]
\label{conj:newmansconjecture}
Let $\Lambda$ be the De Bruijn-Newmann constant. Then $\Lambda\geq 0$.
\end{conjecture}

Note that if both Newman's conjecture \eqref{conj:newmansconjecture} and the Riemann Hypothesis are true, then it must be the case that $\Lambda = 0$. On this, Newman remarked: ``This new conjecture is a quantitative version of the dictum that the Riemann hypothesis, if true, is only barely so'' \cite[Remark 2]{New}. It is remarkable just how precise the bounds on $\Lambda$ are: \cite{SGD} achieved the current best-known bound of $\Lambda \geq -1.14541 \times 10^{-11}.$ To find this bound, Saouter, Gourdon and Demichel build on the work of Csordas, Smith and Varga \cite{CSV}, who use differential equations describing the motion of the zeros under deformation to demonstrate that atypically close pairs of zeros yield lower bounds on $\Lambda$. 

%\green{ADDED THE BEST KNOWN BOUNDS ON $\Lambda$ WITH SOME EXPLANATION. ADDED BOUNDS FOR STOPPLE LATER.}

These ideas have since been translated to many different $L$-functions beyond the Riemann zeta function. Stopple \cite{Sto} showed that there is a real constant $\Lambda_{\text{Kr}}$ analogous to the De Bruijn--Newman constant for quadratic Dirichlet $L$-functions, and Andrade, Chang and Miller \cite{ACM} expanded this to state a version of Newman's conjecture for automorphic $L$-functions. Stopple established bounds on $\Lambda_{\text{Kr}}$ in the case of quadratic Dirichlet $L$-functions; in particular, for the $L$ function corresponding to the quadratic character modulo $D = - 175990483$, $\Lambda_{\text{Kr}} > -1.13 \times 10^{-7}$. The results on lower bounds are extended to automorphic $L$-functions by \cite{ACM}. Most recently, Andrade, Chang and Miller investigated in \cite{ACM} the analogue of the De Bruijn--Newman constant for function field $L$-functions. This is the setting in which we work, so we describe the translation of this framework to the function field setting in Section \ref{sec:newmanfnfield}.

%\green{ADDED BOUNDS FOR STOPPLE.}

%%%%%%%%%%%%%%%%%%%%%%%%%%%%%%%%%%%%%%%%%%%%%%%%%%%%%%%%%%%%%%%%%%%%%%%%%%%%%%%%%%%%%%%%%%%%%%%%%%%%%%%%
%%% SUBSECTION: THIS PAPER %%%%%%
%%%%%%%%%%%%%%%%%%%%%%%%%%%%%%%%%%%%%%%%%%%%%%%%%%%%%%%%%%%%%%%%%%%%%%%%%%%%%%%%%%%%%%%%%%%%%%%%%%%%%%%%
\subsection{This paper}

%\green{ADDED A BREIF EXPOSITION TO QUICKLY INTRODUCE THE FUNCTION FIELDS CASE, THEN EXPLAINED WHAT OUR CONTRIBUTION IS. THIS IS THE OLD INTRODUCTION, BUT MOVED TO FOLLOW  CLASSICAL NEWMAN}

In the function field setting, a similar setup is possible. We will work in this setting, which we develop in Section \ref{sec:newmanfnfield}. It is in this setting that we resolve several versions of the generalized Newman's conjecture first considered by Andrade, Chang and Miller in \cite{ACM}. We recall some of their work in section \ref{sec:previouswork}. We then prove our main result \ref{thm:mainresultintro} in section \ref{sec:results}, where this theorem is stated as theorem \ref{thm:mainresult}.

Over function fields $\F_q[x]$, each quadratic Dirichlet $L$-function $L(s, \chi_D)$ also gives rise to a constant $\Lambda_D$. However, there is very different behavior in this case. For one, it is possible that $\Lambda_D = - \infty$, as in Remark \ref{rmk:lambdaDminisinf}. Therefore, we consider the supremum of the De Bruijn--Newman constants over a family of $L$-functions, so the appropriate analogue of Newman's conjecture becomes the following.

\begin{conjecture}[Newman's Conjecture for Function Fields {\cite[Conjecture 1.8]{ACM}}] Let $\mathcal F$ be a family of $L$ functions over a function field $\F_q[x]$. Then
\begin{equation} \sup_{D \in \mathcal F} \Lambda_D\ =\ 0. \end{equation}
\end{conjecture}

Our main result resolves this conjecture for a wide class of families $\mathcal F$.

\begin{theorem}[Main Result]
\label{thm:mainresultintro} Let $q>0$ be an odd prime power. Let $\mathcal{F}$ be a family of pairs of the form $(D,q)$, where $D\in\mathbb{F}_q[T]$ is monic squarefree polynomial of odd degree at least three, and  $(x^q - x,q)\in\mathcal{F}$. Then if $\Lambda_D$ is the De Bruijn--Newman constant associated to $D$,
\begin{equation}
\sup_{(D,q)\in\mathcal{F}}\Lambda_D\ =\ \max_{(D,q)\in\mathcal{F}}\Lambda_D\ =\ 0 .
\end{equation}
\end{theorem}

To prove this, we first note that $\Lambda_D = 0$ if and only if the $L$-function corresponding to $D$ has a double root. Then we explicitly compute $L(s, \chi_D)$ for $D = x^q - x$ via $\ell$-adic cohomology, the Weil conjectures, and a result of Katz \cite{Katz}. The necessary background for the proof is recalled in Section \ref{sec:weilconjectures}.

Following from Theorem \ref{thm:mainresultintro}, the following conjectures (stated as conjectures \ref{conj:acm1} and \ref{conj:acm2}) of Andrade, Chang and Miller \cite{ACM} are resolved.

\begin{cor}
\label{cor:resolvedconjectures}
 Let $\mathcal F$ be one of the following families of $L$-functions. Then $\sup\limits_{D \in \mathcal F} \Lambda_D = 0$.
\vspace*{-0.7em}
\begin{itemize}
\item $\mathcal F = \{ D \in \F_q[T] \mid \textrm{squarefree, monic, odd degree} \geq 3 \}$;
\item $\mathcal F = \{ D \in \F_q[T] \mid \deg D = 2g + 1,\ 2g+1 = p^k \textrm{ for some prime } p\}$.
\end{itemize}
\end{cor}

%%%%%%%%%%%%%%%%%%%%%%%%%%%%%%%%%%%%%%%%%%%%%%%%%%%%%%%%%%%%%%%%%%%%%%%%%%%%%%%%%%%%%%%%%%%%%%%%%%%%%%%%%%%%%%%%%%%%%%%%%%%%%%%%%%%%
% SETUP FOR NEWMAN IN FUNCTION FIELDS %
%%%%%%%%%%%%%%%%%%%%%%%%%%%%%%%%%%%%%%%%%%%%%%%%%%%%%%%%%%%%%%%%%%%%%%%%%%%%%%%%%%%%%%%%%%%%%%%%%%%%%%%%%%%%%%%%%%%%%%%%%%%%%%%%%%%%

\section{Setup for a Newman's Conjecture over Function Fields}\label{sec:newmanfnfield}

In \cite{ACM}, the authors find many analogues between the number field and function field versions of Newman's conjecture. The appropriate replacement for $\ZZ$ is $\F_q[T]$, the ring of polynomials with coefficients in $\F_q$ (the finite field with $q$ elements), where $q$ is a power of a prime. We do not consider fields of characteristic two.

\begin{mydef}[Andrade, Chang, Miller {\cite[Definition 3.1]{ACM}}]
Let $q$ be an odd prime power and let $D\in\mathbb{F}_q[T]$. We say that $(D,q)$ is a \textit{good pair} or simply that $D$ is \textit{good} if
\begin{enumerate}
\item $D$ is monic and square-free,
\item $\deg D$ is odd, and
\item $\deg D\geq 3$.
\end{enumerate}
\end{mydef}

The rationale for these assumptions is elucidated in \cite[Remark 3.2]{ACM}. In short, we assume squarefree and monic because this corresponds to the fundamental discriminants in the number field setting, and we assume $q$ is odd because we are not considering the characteristic 2 case, in which everything is a perfect square.  Instead of the Riemann zeta function in this setting, we have the following.

\begin{mydef}
Let $D$ be good. We define the $L$-function associated to $D$ to be
\begin{equation}
\label{eq:lfunc}
L(s,\chi_D)\ :=\ \sum_{f\textrm{{\rm\ monic}}}\chi_D(f) N(f)^{-s},
\end{equation}
where $N(f)$ is the norm of $f$, $N(f):=q^{\deg f}$ and $\chi_D$ is the Kronecker symbol: $\chi_D(f) := \left(\frac{D}{f}\right)$.
\end{mydef}
If we collect terms in \eqref{eq:lfunc}, we have
\begin{equation}
L(s,\chi_D)\ =\ \sum_{n\geq 0} c_n\left(q^{-s}\right)^n,
\end{equation}
where
\begin{equation}
c_n\ =\ \sum_{\substack{f\textrm{ monic}\\ \deg f = n}}\chi_D(f).
\end{equation}
These coefficients vanish for $n \geq \deg D$, and $L(s, \chi_D)$ is a polynomial of degree exactly $\deg D - 1$ in $q^{-s}$. Setting $g$ to be the genus of the hyperelliptic curve $y^2 = D(x)$, (so $g = (\deg D - 1) / 2$), we complete $L(s, \chi_D)$ as we completed the Riemann zeta function in \eqref{eq:completedzeta} in the classical setting. Set
\begin{equation}
\xi(s, \chi_D)\ :=\ q^{gs} L(s, \chi_D),
\end{equation}
so that $\xi(s, \chi_D)$ satisfies the nice functional equation $\xi(s, \chi_D) = \xi(1-s, \chi_D)$. Then in analogy to \eqref{eq:changevariableszeta}, define
\begin{equation}
\Xi(s, \chi_D)\ :=\ \xi\left( \frac{1}{2} + i \frac{x}{\log q}\ , \chi_D \right)\ =\ \Phi_0 + \sum_{n=1}^g \Phi_n (e^{inx} + e^{-inx})\ ,
\end{equation}
where $\Phi_n\ =\ c_{g-n} q^{n/2} = c_{g+n} q^{n/2}$. Note that $\Phi_n$ is the Fourier transform of $\Xi_t$ in this case, which is a Fourier transform on the circle instead of the real line.

\begin{mydef} The \emph{deformed $L$-function} $\Xi_t(x, \chi_D)$ is
\begin{equation}
\label{eq:deformedLfunc}
\Xi_t(x, \chi_D)\ :=\ \Phi_0 + \sum_{n=1}^g \Phi_n e^{tn^2} (e^{inx} + e^{-inx}).
\end{equation}
\end{mydef}

Andrade, Chang and Miller established the existence of a De Bruijn--Newman constant $\Lambda_D$ for each good $D \in \F_q[T]$ \cite[Lemma 3.4]{ACM}. The behavior of this constant $\Lambda_D$ is in some ways similar to that of the classical De Bruijn--Newman constant, as the following lemma shows.

\begin{lemma}[{\cite[Lemma 3.2]{ACM}}]
\label{lem:doublezeroes}
Let $(q, D)$ be a good pair. Let $t_0 \in \RR$. If $\Xi_{t_0}(x, \chi_D)$ has a zero $x_0$ of order at least $2$, then $t_0 \leq \Lambda_D$ .
\end{lemma}

However, there are differences between function and number fields. As there is a proof of the  Riemann Hypothesis in function fields, $\Lambda_D \leq 0$. Furthermore, for many $D$ we can actually get the strict inequality $\Lambda_D < 0$ because of the existence of the following partial converse to Lemma \ref{lem:doublezeroes} above.

\begin{lemma}
If $\Xi_0(x, \chi_D)$ does not have a double zero, then $\Lambda_D < 0$.
\end{lemma}

In fact, since there are $L$-functions $\Xi_0(x, \chi_D)$ with only real zeros, there is the possibility of $\Lambda_D = -\infty$, as in the following remark.

\begin{remark}
\label{rmk:lambdaDminisinf}
In the function field setting, we may have that $\Lambda_D = - \infty$.  Indeed, \cite[Remark 3.10]{ACM} supplies a counterexample: $D = T^3 + T \in \F_3[T]$ has $\Lambda_D = -\infty$.
\end{remark}

In light of this remark, we can see that Newman's conjecture, if directly translated to the function field setting, is false. Thus, different versions of Newman's conjecture are necessary. Most generally, the modified Newman's conjecture is the following.

\begin{conjecture}[Newman's Conjecture for Function Fields {\cite[Conjecture 1.8]{ACM}}] Let $\mathcal F$ be a family of $L$ functions over a function field $\F_q[x]$. Then
\begin{equation} \sup_{D \in \mathcal F} \Lambda_D\ =\ 0. \end{equation}
\end{conjecture}

Regarding the De Bruijn--Newman constant in the function field setting, Andrade, Chang and Miller made the following conjectures for specific families.

%SOME OF THE CONJECTURES MADE IN ACM
\begin{conjecture}[Andrade, Chang, Miller {\cite[Conjecture 3.14]{ACM}}]
\label{conj:acm1}
Fix $q$ a power of an odd prime. Then
\begin{equation}
\sup_{(q,D)\textrm{\ {\rm good}}} \Lambda_D\ \geq\ 0.
\end{equation}
\end{conjecture}
\begin{conjecture}[Andrade, Chang, Miller {\cite[Conjecture 3.15]{ACM}}]
\label{conj:acm2}
Fix $g\in\NN$. Then
\begin{equation}
\sup_{\substack{\deg D = 2g + 1 \\ (q,D)\textrm{\ {\rm good}}}}\Lambda_D\geq 0.
\end{equation}
\end{conjecture}
\begin{conjecture}[Andrade, Chang, Miller {\cite[Conjecture 3.16]{ACM}}]
\label{conj:acm3}
Fix $D\in\mathbb{Z}[T]$ square-free. For each prime $p$, let $D_p\in\mathbb{F}_p[T]$ be the polynomial obtained from reducing $D\pmod p$. Then
\begin{equation}
\sup_{(p,D_p)\textrm{\ {\rm good}}} \Lambda_{D_p}\geq 0.
\end{equation}
\end{conjecture}
The last of these three conjectures, Conjecture \ref{conj:acm3}, was resolved in \cite[Theorem 3.19]{ACM} for the case where $\deg D = 3$. This result is briefly reviewed in the following subsection.

%%%%%%%%%%%%%%%%%%%%%%%%%%%%%%%%%%%%%%%%%%%%%%%%%%%%%%%%%%%%%%%%%%%%%%%%%%%%%%%%%%%%%%%%%%%%%%%%%%%%%%%%%%%%%%%%%%%%%%%%%%%%%%%%%%%%
% PREVIOUS RESULTS %
%%%%%%%%%%%%%%%%%%%%%%%%%%%%%%%%%%%%%%%%%%%%%%%%%%%%%%%%%%%%%%%%%%%%%%%%%%%%%%%%%%%%%%%%%%%%%%%%%%%%%%%%%%%%%%%%%%%%%%%%%%%%%%%%%%%%
\subsection{Previous Results}\label{sec:previouswork}

In this section, the work of Andrade, Chang and Miller in \cite{ACM} to prove Newman's conjecture for families given by elliptic curves is described.
Fix a square-free polynomial $\mathcal D \in \ZZ[T]$ of degree $3$ and for each prime $p$, let $D_p \in \F_p[T]$ be the polynomial obtained by reducing $D \pmod{p}$. In \cite{ACM} Newman's conjecture is proved for the family $\mathcal F = \{ D_p \}$.

\begin{theorem}[{\cite[Theorem 3.19]{ACM}}]
For $\mathcal F$ defined above, $\sup_{D \in \mathcal F} \Lambda_D = 0$.
\end{theorem}

The proof uses the fact that
\begin{equation}
\label{eq:explicitcomputation}
\Xi_t(x, \chi_{D_p}) \ = \  -a_p(\mathcal D) + 2 \sqrt{p} e^{t} \cos x,
\end{equation}
where $a_p(\mathcal D)$ is the trace of Frobenius of the elliptic curve $y^2 = \mathcal D(T)$. From this, we can deduce
\begin{equation}\label{eq:explicitlambda}
\Lambda_{D_p} \ = \  \log \frac{|a_p(\mathcal D)|}{2\sqrt{p}}.
\end{equation}

Finally, the recent proof of Sato--Tate for elliptic curves without complex multiplication \cite{satotate1, satotate3, satotate2, B-LGHT} implies that there exists a sequence of primes $p_1, p_2, \ldots$ such that
\begin{equation}\label{eq:seq-primes-to-1}
\lim_{n \to \infty}
\frac{a_{p_n}(\mathcal D)}{2\sqrt{p_n}}
\ \to \ 1.
\end{equation}
Hence $\Lambda_{D_{p_n}} \to 0$.

%\begin{theorem}[Newman's conjecture for fixed $\mathcal D$, $\deg \mathcal D = 3$]
%Let $\mathcal D \in \ZZ[T]$ be square-free with $\deg \mathcal D = 3$. Then $\sup_{p} \Lambda_{D_p} = 0$.
%\end{theorem}

\begin{remark}
It should be noted that \eqref{eq:explicitlambda} holds not only for $p$ prime, but also when $p$ is replaced by a prime power $q$. We make use of this more general explicit form of $\Lambda$ later on.
\end{remark}

%\green{ADDED A LITTLE MORE EXPLANATION ABOUT COMPUTING $\Lambda_{D_p}$ EXPLICITLY.}
\begin{remark}
The proof relied on the fact that $\Lambda_{D_p}$ could be computed explicitly, which is made possible by the fact that $D$ has genus $g = 1$, so there are only two terms to consider when computing \ref{eq:deformedLfunc}.  When $g \geq 2$, then $\Xi_t$ contains multiple $e^{t}$ terms and therefore multiple $\cos nx$ terms, making it much harder to find the explicit expression of $\Lambda_{D_p}$.
\end{remark}

\begin{remark}
We are not aware of any proof of the existences of a sequence of primes for which \eqref{eq:seq-primes-to-1} holds without appealing to proven Sato--Tate laws; it would be interesting to have an elementary proof of such a statement.
\end{remark}

The previous two remarks above suggest that it would be difficult to prove results for $\deg \mathcal D \geq 5$ using the same methods.

%\begin{lemma}
%Let $D\in\mathbb{F}_q[T]$ be square-free and degree $3$. Then
%$$
%\Lambda_{D} = \log\frac{\abs{a_q(D)}}{2\sqrt{q}}
%$$
%\end{lemma}

%%%%%%%%%%%%%%%%%%%%%%%%%%%%%%%%%%%%%%%%%%%%%%%%%%%%%%%%%%%%%%%%%%%%%%%%%%%%%%%%%%%%%%%%%%%%%%%%%%%%%%%%%%%%%%%%%%%%%%%%%%%%%%%%%%%%
% HASSE-WEIL ZETA FUNCTION AND THE WEIL CONJECTURES %
%%%%%%%%%%%%%%%%%%%%%%%%%%%%%%%%%%%%%%%%%%%%%%%%%%%%%%%%%%%%%%%%%%%%%%%%%%%%%%%%%%%%%%%%%%%%%%%%%%%%%%%%%%%%%%%%%%%%%%%%%%%%%%%%%%%%
\section{The Hasse-Weil Zeta Function and the Weil Conjectures}\label{sec:weilconjectures}

Let $X/\mathbb{F}_q$ be a curve, $q$ a power of a prime $p$ as always.

\begin{mydef}
Let $N_m := \#X(\mathbb{F}_{q^m})$. The \textit{Hasse-Weil zeta function} of the curve $X$ is defined by the formal power series
\begin{equation}
Z(X,s)\ :=\ \exp\left(\sum_{m\geq 1}\frac{N_m}{m}\left(q^{-s}\right)^m\right).
\end{equation}
\end{mydef}

\begin{remark}
Most of the definitions and results in this section hold in much greater generality than stated here, but for ease of exposition we will only state results in the generality required for our applications.
\end{remark}

It isn't immediately clear from the definition that these zeta functions are useful objects to consider, but the following canonical example illustrates that in fact the zeta function contains information about the geometry of $X$.

\begin{example}
Let $X=\P^1(\F_q)$. It follows that $N_m = q^m + 1$, so we obtain the following expression for the zeta function after setting $T=q^{-s}$:
\begin{equation} Z(X,s)\ =\ \frac{1}{(1-T)(1-qT)}. \end{equation}\\
The two  terms linear in $T$ in the denominator reflect the fact that $H^0_{\acute{e}t}(\P^1, \Ql)$ and $H^2_{\acute{e}t}(\P^1, \Ql)$ are one-dimensional. The lack of a linear term in the numerator reflects the fact that $H^1_{\acute{e}t}(\P^1, \Ql)=0$.
\end{example}

Henceforth, we set $T = q^{-s}$ unless stated otherwise. There is a collection of theorems called the Weil conjectures (although they have now been proven) which make more precise the relationship between the geometry of $X$ and its zeta function. The Weil conjectures were first stated for algebraic curves by Artin, and were proven later by Dwork and Deligne. We now state the subset of the Weil conjectures relevant for this paper.

\begin{theorem}\label{thm:Weilconjectures}
Let $X / \mathbb{F}_q$ be a nonsingular projective curve. Then the Hasse-Weil zeta function  $Z(X,s)$ of $X$ has the form
\begin{equation}
Z(X,s)\ =\ \frac{P(T)}{(1 - T)(1 - qT)},\qquad P\in\ZZ[T].
\end{equation}
Moreover,
\begin{enumerate}
\item $\deg P = 2g$, where $g$ is the genus of the curve $X$, and
\item $P$ factors as $\prod_{i = 1}^{2g}(1 - \alpha_iT)$. For all $i$, $| \alpha_i |= q^{1/2}$.
\end{enumerate}
\end{theorem}
By putting these results together and unwinding the definition of $Z(X,s)$ as a generating function, we obtain the following useful result.
\begin{cor}
\label{cor:pointcounting}
Let $X$ be a nonsingular projective curve of genus $g$. Then the numbers $\alpha_1, \ldots \alpha_{2g}$ coming from the zeta function satisfy
\begin{equation}
 \label{eq:numpoints}
\# X(\F_q^m)\ =\ 1+ q^m - \sum_{i}^{2g} \alpha_i^m.
\end{equation}
\end{cor}

\begin{remark}
\label{rmk:pointcounting}
We have the following useful application of Corollary \ref{cor:pointcounting}. Let $E/\mathbb{F}_p$ be an elliptic curve such that $\#E(\mathbb{F}_p) = p+1$. Recall that elliptic curves have genus $1$. By \eqref{eq:numpoints} we have $\alpha_1+ \alpha_2 = 0$. We also have $P(T) = (1-\alpha_1T)(1-\alpha_2T) \in \ZZ[T]$, so that $\alpha_1 \alpha_2 \in \ZZ$. Since $|\alpha_i | = \sqrt{p}$, we have $\alpha_1 \alpha_2 = \pm p$. However, the first condition implies that we must have (after possibly reordering) $\alpha_1= i \sqrt{p}$, $\alpha_2 = -i \sqrt{p}$. Now we compute
\begin{equation}
\# E(\F_p^2)\ =\ p^2 + 1 - (i\sqrt{p})^2 - (-i\sqrt{p})^2\ =\ p^2+2p+1.
\end{equation}
\end{remark}

The computation in Remark \ref{rmk:pointcounting} and generalizations thereof will be very important later for proving particular cases of Newman's conjecture in families by constructing particular elliptic curves $E / \F_p$ with $p+1$ points. The condition of having $p^2+2p+1$ points over $\F_{p^2}$ is significant because Corollary ~\ref{cor:pointcounting} implies that this number is as large as possible for a curve of genus $1$ over $\F_q$.

\begin{mydef}
We say a curve $X$ with $q + 2\sqrt{q} + 1$ points over $\F_q$ is \emph{maximal} over $\F_q$. Similarly, $X$ is \emph{minimal} if it has $q - 2\sqrt{q}+1$ points over $\F_q$.
\end{mydef}

\begin{remark}It is clear that a curve can only be maximal (or minimal) over $\F_q$ when $q$ is a square. This will be important for proving cases of Newman's conjecture in families.
\end{remark}

Corollary \ref{cor:pointcounting} allows us to prove a special case of Newman's conjecture using the explicit formula for $\Lambda_D$ found in \cite{ACM}, when $y^2 = D(x)$ is an elliptic curve. We prove this result by explicitly relating our $L$-function $L(s,\chi_D)$ to the zeta function of the curve $y^2 = D(x)$.

\begin{prop}
$L(s,\chi_D)$ is the numerator of the zeta function $Z(X,s)$, where $X$ is the curve defined by $y^2 = D(x)$. More precisely, $Z(X,s) = Z(\mathbb{P}^1,s)L(s,\chi_D)$.
\end{prop}

\begin{proof}
In this proof, we imitate the proof that the Dedekind zeta function of a quadratic field $\QQ(\sqrt{d})$ factors as $\zeta(s) L(s, \chi_D)$. In this case, the quadratic field is $F := \F_p(t)\left(\sqrt{D(t)}\right)$. Recall we have
\begin{align}
L(s, \chi_D) = \sum_{f\ {\rm monic}} \chi_D(f) N(f)^{-s}
 = \prod_{\substack{g\text{ monic,} \\ \text{irreducible}}} (1-\chi_D(g)) N(g)^{-s}.
\end{align}
Note that $\chi_D(g) = 1$ if and only if $g$ splits in $F$, and $\chi_D(g) = -1$ if and only if $g$ is inert in $F$, and finally $\chi_D = 0$ if and only if $g \mid D$. Thus we have the following factorization of $L(s, \chi_D)$
\begin{align}
L(s, \chi_D)\ & =\ \prod_{g\text{ split}} (1-N(g)^{-s})^2 \prod_{g\text{ inert}} (1-N(g)^{-2s})  \nonumber\\[1em]
& = \prod_{\substack{g\text{ monic,} \\ \text{irreducible}}} (1-N(g)^{-s}) \prod_{\substack{g\text{ monic,} \\ \text{irreducible}}} (1-\chi_D(g)N(g)^{-s})  \nonumber\\[1em]
& = \prod_{\substack{g\text{ monic,} \\ \text{irreducible}}} (1-N(g)^{-s}) \sum_{g\text{ monic}} \chi_D(g) N(g)^{-s} \nonumber\\[1em]
& = \prod_{\substack{g\text{ monic,} \\ \text{irreducible}}} (1-N(g)^{-s}) L(s, \chi_D). \label{genfunc}
\end{align}
It remains to be shown that $\prod_{g\text{ irreducible}} (1-N(g)^{-s})$ is in fact $Z(\P^1, s)$. Recall
\begin{align}
Z(\P^1 , s ) &\ =\ \frac{1}{(1-T)(1-qT)} \nonumber\\[1em]
&\ =\ (1+T+T^2+ \cdots)(1+qT+q^2 T^2 + \cdots)
\prod_{\substack{g\text{ monic,} \\ \text{irreducible}}} (1-N(g)^{-s})\nonumber\\
& \ =\ \sum_{g\text{ monic}} N(g)^{-s} = \sum_{n=1}^{\infty} c_n T^n,
\end{align}
where the coefficients $c_n$ are the number of monic polynomials in $\F_q(t)$ of degree $n$. This generating function and the generating function in \eqref{genfunc} are easily seen to be equal, which finishes the proof.
\end{proof}

Now that we can realize our $L$-function $L(s,\chi_D)$ as part of the zeta function $Z(X,s)$, the Weil conjectures tell us valuable information about the behavior of the roots of $L$. In particular, we will be able to prove that certain curves have a double root (in fact, a root of multiplicity $g$).

%%%%%%%%%%%%%%%%%%%%%%%%%%%%%%%%%%%%%%%%%%%%%%%%%%%%%%%%%%%%%%%%%%%%%%%%%%%%%%%%%%%%%%%%%%%%%%%%%%%%%%%%%%%%%%%%%%%%%%%%%%%%%%%%%%%%
% Results %
%%%%%%%%%%%%%%%%%%%%%%%%%%%%%%%%%%%%%%%%%%%%%%%%%%%%%%%%%%%%%%%%%%%%%%%%%%%%%%%%%%%%%%%%%%%%%%%%%%%%%%%%%%%%%%%%%%%%%%%%%%%%%%%%%%%%

\section{Families of Curves Satisfying Newman's Conjecture}
\label{sec:results}

We are now ready to prove our main result, Theorem \ref{thm:mainresultintro}. It is restated below as Theorem \ref{thm:mainresult}. Throughout this section we assume familiarity with the theory of \'etale cohomology (specifically $\ell$-adic cohomology) of projective curves; development of this subject can be found in J.S. Milne's book \cite{Milne}.
\begin{theorem}
\label{thm:mainresult}
Let $\mathcal{F}$ be a family of pairs of the form $(D,q)$, where $D\in\mathbb{F}_q[T]$ is monic squarefree polynomial of odd degree at least three, and  $(x^q - x,q)\in\mathcal{F}$. Then
\begin{equation}
\sup_{(D,q) \in\mathcal{F}}\Lambda_D\ =\ \max_{(D,q)\in\mathcal{F}}\Lambda_D\ =\ 0 .
\end{equation}
\end{theorem}

To prove this theorem, we need the following key lemma.

\begin{lemma}
$L(s,\chi_D)$ has a double root if and only if $\Lambda_D = 0$.
\end{lemma}
\begin{proof}
If $L(s, \chi_D)$ has a double root, then the fact that $\Lambda_D = 0$ follows from \cite[Lemma 3.22, Remark 3.23]{ACM}. The converse is a consequence of \cite[Lemma 3.11]{ACM}.
\end{proof}
The immediate consequence of this lemma is the following.
\begin{cor}
If $L(s,\chi_D)$ has a double zero, then Newman's conjecture is true for any family $\mathcal{F}$ containing $D$. \end{cor}

To show that $\sup_{(D,q) \in \mathcal F} \Lambda_D = 0$, it suffices to find $D$ such that $L(s, \chi_D)$ has a double root. In this case, $\Lambda_D = 0$, so the supremum is actually a maximum.

\begin{prop}
\label{prop:maximizer}
Let $D\in\mathbb{F}_q[x]$ be given by $D(x) = x^q - x$. Then $L(s,\chi_D)$ has a double root.
\end{prop}
\begin{remark}
In fact, $L(s,\chi_D)$ has a root of order $g$, and $L(s,\chi_D)$ is explicitly given by $L(s,\chi_D) = (T^2q\pm 1)^g$.
\end{remark}
\begin{proof}[Proof of Proposition \ref{prop:maximizer}]
The curve $X : y^2 = x^q - x$ carries an action of $G=\mathbb{F}_q$ by $\F_q$-linear automorphisms. That is, the action of $\F_q$ commutes with the Frobenius map. For $a \in \F_q$, the action is defined by
\begin{equation}
a \cdot(x,y)\ =\ (x+a, y).
\end{equation}\\
By functoriality, the cohomology groups $\HiX$ carry an action of $G$ as well. In certain nice cases, $\HstarX$ splits up into distinct irreducible representations of $G$ (i.e., is multiplicity free). In this case, this means $\HstarX$ is a sum of characters of $G$. Since the action of Frobenius commutes with the action of $G$, we have a well-defined action of Frobenius on $\HstarX[\chi]$, the $\chi$-isotypic component for $\chi$ a character of $G$. Assuming the multiplicity-free hypothesis, these spaces are $1$-dimensional and Frobenius acts as a scalar, which is therefore a Frobenius eigenvalue. Since the only information needed to construct the zeta function are the Frobenius eigenvalues, we can construct the zeta function if we can understand $\HstarX$ as a representation of $G$ and how Frobenius acts. It is a result of Nick Katz (restated below as Theorem \ref{thm:Katz}) that gives conditions for the above to be true and also gives the Frobenius eigenvalues explicitly as Gauss sums. The following theorem gives us the decomposition of $\HstarX$ and the Frobenius eigenvalues, which finishes the proof.
\end{proof}

\begin{theorem}[Katz \cite{Katz}] \label{thm:Katz}
Let $X \slash \F_q$ be projective and smooth, and $G$ a finite group acting on $X$ by $\F_q$-linear automorphisms, and $\rho$ an irreducible complex (or $\ell$-adic) representation of $G$. Define
\begin{equation}
S(X / \F_q, \rho, n)\ :=\ \frac{1}{\# G} \sum_{g \in G} \textnormal{Tr}(\rho(g))\, \#\textnormal{Fix}(\frob_p \circ g^{-1}).
\end{equation}
Then the following are equivalent.\\
\\
(1) The multiplicity of $\rho$ is one in $\HizeroX$ and zero in $\HiX$ for $i \neq i_0$. \\
\\
(2) For all $n \geq 1$, we have
\begin{align}
| S(X \slash F_q, \rho, n) |\ =\ (\sqrt{q})^{i_0 n}.
\end{align}
(3) $\frob$ acts on $\HiX$ by the scalar $(-1)^{i_0} S(X \slash \F_q, \rho, 1)$.
\end{theorem}
Now we compute the sums $S(X \slash \F_q, \chi, n)$ as $\chi$ ranges over the characters of $G=\F_q$ (since $G$ is abelian its irreducible representations are characters). The characters of $\F_q$ are parametrized by $\F_q$ itself, as they take the form $\chi_a$ where $\chi_a (b) = \zeta_p^{Tr(ab)}$, where $\zeta_p$ is a primitive $p$\textsuperscript{th} root of unity.
\begin{theorem}
The conditions in Theorem \ref{thm:Katz} hold for $X: y^2=x^q-x / \F_q$. Let $q=p^r$ and $p^*= \left( \frac{-1}{p} \right)p$. As a representation of $G=\F_q$
\begin{equation}
\Hone\ =\ \sum_{\chi\textrm{\ {\rm nontrivial}}} \chi.
\end{equation}
The Frobenius eigenvalues are $\sqrt{p^*}^{r}$ and $-\sqrt{p^*}^{r}$, both with multiplicity $\frac{q-1}{2}$.
\end{theorem}
\begin{proof}
Our projectivized curve $X$ is given by $y^2 z^{q - 2} = x^q - x z^{q - 1}$, which is easily checked to be smooth. We have the following equality:
\begin{align}
S(X, \chi, n ) \ =\  \frac{1}{q} \sum_{a \in \F_q} \chi(a) \#Fix(\frob_q^n \circ [-a]).
\end{align}
Now we must determine when a point $(x,y)$ is fixed by $\frob_q^n \circ[-a]$. This is when $y^{q^n} = y$ and $x^{q^n} -x = a$. That means $y \in \F_{q^n}$ and $Tr(x^q-x)=a$, where $Tr: \F_{q^n} \rightarrow \F_q$ is the field theoretic trace. For a fixed $y \in \F_{q^n}$, when $Tr(y^2) = a$, we have $q$ fixed points corresponding to the $q$ distinct solutions to $x^q-x=a$, otherwise we have $0$ fixed points. Define
\begin{equation}
I(y,a)\ =\ \left\{
        \begin{array}{ll}
            1 & $if$ \  Tr(y) = a \\
            0 & $otherwise$.
        \end{array}
    \right.
\end{equation}
Then we have
\begin{align}
S(X, \chi, n) &\ =\ \frac{1}{q} \sum_{a \in \F_q} \chi(a) \sum_{y \in \F_{q^n}} q I(y,a) \nonumber\\
&\ =\ \sum_{y \in \F_{q^n}} \sum_{a \in \F_q} \chi(a) I(y,a) \nonumber\\
&\ =\ \sum_{y \in \F_{q^n}} \chi(Tr(y^2)) \nonumber\\
&\ =\ \sum_{y \in \F_{q^n}} \chi(Tr(y)) \left( \frac{N_m(y)}{q} \right).
\end{align}
Note here $\left( \frac{ \cdot}{q} \right)$ is the $\F_q$-Legendre symbol. Now we apply the Hasse-Davenport relation which says precisely that the above ``Gauss sums" are (up to sign) just powers of Gauss sums over $\F_q$. More precisely, it tells us that
\begin{align}
\left|S(X, \chi, n) \right|\ =\ \left| \sum_{y \in \F_q} \chi(y) \left( \frac{y}{q} \right) \right|^n.
\end{align}
If $\chi$ is nontrivial, it is a well-known result that the inner sum has magnitude $\sqrt{q}$. That is,
\begin{equation}
|S(X, \chi, n)|\ =\ \left\{
        \begin{array}{ll}
            (\sqrt{q})^n & $if$ \ \chi \ $is trivial$\\
            q^n & $otherwise$.
        \end{array}
    \right.
\end{equation}
By Theorem \ref{thm:Katz} that means that
\begin{equation}
\Hone\ =\ \sum_{\chi \text{ nontrivial}} \chi, \
\Htwo\ =\ \chi_{{\rm triv}}
\end{equation}
as representations of $\F_q$. The reason $\Hzero$ doesn't appear is because in order for our results to be true, we must take compactly supported cohomology, which forces $\Hzero= 0$, since $X$ isn't compact. By the equivalent condition of Theorem \ref{thm:Katz}, we know that the Frobenius eigenvalues on $\Hone$ are the sums $(-1)S(X, \chi, 1)$. We also immediately verify that Frobenius acts on $\Htwo$ by the scalar $q$ in accordance with the Weil conjectures. To compute $S(X, \chi, 1)$, we write $q=p^r$, and $\chi = \chi_a$ for some $a \in \F_q$. Applying the Hasse-Davenport relation again and using the computation of the standard quadratic Gauss sum over $\F_p$ gives
\begin{equation}
S(X, \chi_a, 1)\ =\ \left\{
        \begin{array}{ll}
            (-1)^{r+1} (\sqrt{p^*})^r & $if$ \ a \ $is a quadratic residue$ \\
            (-1)^{r} (\sqrt{p^*})^r & $otherwise$.
        \end{array}
    \right.
\end{equation}
\end{proof}

\begin{cor}
Let $\mathcal{F} = \{D\in\mathbb{F}_q[T]\mid D\textrm{ monic, squarefree, of odd degree} \geq 3\}$. Then
\begin{equation}
\sup_{D\in\mathcal{F}}\Lambda_D = 0.
\end{equation}
\end{cor}

\begin{cor} %!!! needs to be made precise since fields that they're over vary
Let $\mathcal{F} = \{D\mid \deg D = 2g + 1, 2g + 1 = p^k\textrm{ for some prime }p\}$. Then
\begin{equation}
\sup_{D\in\mathcal{F}}\Lambda_D = 0.
\end{equation}
\end{cor}

In addition to the above two corollaries, we can also show that the following family satisfies Newman's conjecture.

\begin{theorem}
Let $D\in\ZZ[T]$ be a square-free monic cubic polynomial. Then there exists a number field $K/\QQ$ such that
\begin{equation}
\sup_{\fp\subseteq\OK}\Lambda_{D_\fp} = 0,
\end{equation}
where $D_\fp$ denotes the reduction of $D$ modulo the prime ideal $\fp$.
\end{theorem}
\begin{proof}
It suffices to produce a single prime $\pi$ so that $a_{\pi} (D) = 2\sqrt{p^2} $, so that $\Lambda_{D_\pi} = 0$ by the previous lemma. If we can find $p \in \ZZ$ inert in $K$ with $a_p (D) = 0$, then for $\pi = p \OK$ we have $a_\pi (D) = 2 \sqrt{p^2}$ by the Weil conjectures. Thus $\Lambda_{D_\pi} = 0$, which gives the result. It is important to note that we don't need to take the supremum over all $\pi$, since any $\pi$ as constructed above attains the supremum.

For all but finitely many $p$, we can reduce $y^2=D(x)$ mod $p$ and thereby obtain an elliptic curve over $\F_p$. For these $p$, the condition that $a_p(D) = p+1$ can be rephrased as saying that $p$ is a supersingular prime for $E$ (as long as $p>5$). It is a theorem of Noam Elkies \cite[Theorem 1]{elkies} that for $E \slash \QQ$ an elliptic curve, and any finite set of primes $S$, we can find a supersingular prime for $E$ outside of $S$. This result uses the theory of complex multiplication of elliptic curves. Now we need only choose $d \in \ZZ$ and $p$ a supersingular prime so that $\left( \frac{d}{p} \right)=-1$, which is easily accomplished since we are free to choose $d$ squarefree belonging to a class which is a quadratic non-residue mod $p$. Thus $p$ is inert in $K$, and the proof is complete.
\end{proof}
\begin{remark}\label{heyslugger}
Since we can find a supersingular prime of $E$ outside of any finite set, we might hope to prove something stronger, namely, we might want to fix $K$ beforehand and hope that the collection of supersingular primes of $E$ contains a prime inert in $K$. Unfortunately, this statement can fail for $K$ quadratic.
\end{remark}

%\green{THIS COUNTEREXAMPLE ACTUALLY CAME FROM PRIVATE CORRESPONDENCE BETWEEN DYLAN AND ELKIES, SO I MADE THAT EXPLICITLY CLEAR. THIS ISN'T FROM ANYTHING THAT ELKIES HAS PUBLISHED, BUT WE STILL WANT TO ATTRIBUTE IT TO HIM.}

The following counterexample was suggested to us in correspondence with Noam Elkies.

\begin{example}[Elkies] Consider $X_0(11)$, an elliptic curve over $\QQ$ with $5$-torsion and good reduction away from $11$. For $p \neq 11$, the $5$-torsion points remain distinct mod $p$, giving
\begin{equation}
\# X_0(11)(\F_p)\ \equiv\ 0 \pmod{5}.
\end{equation}
Thus if $p$ is supersingular for $X_0(11)$, we must have $p \equiv 4 \pmod{5}$, which forces $p$ to split in $K=\QQ(\sqrt{5})$. Thus, since supersingularity is equivalent to $a_p=p+1$ only for $p>5$, we check $p=2,3$ and $5$ seperately and see that $a_p \neq 0$. Thus we've shown that for $E=X_0(11)$ and $K=\QQ(\sqrt{5})$, we cannot find a prime $\pi \subset \OK$ so that $\Lambda_{D_\pi} = 0$.
\end{example}

\begin{remark} For a representation theoretic explanation of this counterexample, recall that for elliptic curves $E \slash \QQ$, and primes $\ell$ not dividing the conductor of $E$ we can consider the mod $\ell$ representation attached to E:
\begin{equation}
\overline{\rho_{E, \ell}}: G_{\QQ} \rightarrow GL_2(\F_\ell).
\end{equation}
For $q \not\mid \ell$ and the conductor of $E$, we have that $Tr(\overline{\rho_{E, \ell}}(Frob_q)) \equiv a_q \pmod{\ell}$. When we're in the case that the mod $\ell$ representation is surjective, we can find $q$ so that $a_q \equiv 0 \pmod{\ell}$, which is necessary but not sufficient for $a_q=0$, which we want in order to construct maximal curves. In the case of $E=X_0(11)$, we can compute using SAGE that the mod $5$ representation is not surjective, which helps explain why $a_p \equiv 0 \pmod{5}$ can't be attained. It follows from Serre's open image theorem that the mod $\ell$ representation is surjective.
One avenue for future research is to use surjectivity of the mod $\ell$ representation to strengthen the above theorem as much as possible.
\end{remark}

%%%%%%%%%%%%%%%%%%%%%%%%%%%%%%%%%%%%%%%%%%%%%%%%%%%%%%%%%%%%%%%%%%%%%%%%%%%%%%%%%%%%%%%%%%%%%%%%%%%%%%%%%%%%%%%%%%%%%%%%%%%%%%%%%%%%
% BIBLIOGRAPHY
%%%%%%%%%%%%%%%%%%%%%%%%%%%%%%%%%%%%%%%%%%%%%%%%%%%%%%%%%%%%%%%%%%%%%%%%%%%%%%%%%%%%%%%%%%%%%%%%%%%%%%%%%%%%%%%%%%%%%%%%%%%%%%%%%%%%

\ \\

\end{document}